\documentclass[12pt,twoside]{article}
\usepackage{amsfonts,amssymb,amsmath,amsthm}
\usepackage{color}
\usepackage{graphicx}
\usepackage{psfrag}
\usepackage{enumerate}
\usepackage{datetime}
\usepackage{fancyhdr}
 
\newtheorem*{theorem*}{\bf Theorem}
\newtheorem*{addendum*}{\bf Addendum}

\newtheorem{definition}{\bf Definition}[section]
\newtheorem{theorem}[definition]{\bf Theorem}
\newtheorem{lemma}[definition]{\bf Lemma}

\newtheorem{proposition}[definition]{\bf Proposition}
\newtheorem{corollary}[definition]{\bf Corollary}
\newtheorem{remark}[definition]{\bf Remark}

\newtheorem*{conjecture*}{\bf Conjecture}

\newcommand\BA{\mathbf A}\newcommand\BB{\mathbf B}\newcommand\BC{\mathbf C}\newcommand\BD{\mathbf D}

\def\QQ{\mathbb Q}

\def\ZZ{\mathbb Z}
\def\RR{\mathbb R}

\def\cO{\mathcal O}

\def\cA{\mathcal A}
\def\cB{\mathcal B}

\newcommand{\Diff}{{\operatorname{Diff}}}
\newcommand\Id{\operatorname{Id}}

\renewcommand{\color}[1]{\bf}

\definecolor{darkgreen}{rgb}{0.0, 0.3, 0.0}

\usepackage{soul}

\date{\today}
\title{Local perturbations of conservative $C^1$ diffeomorphisms}

\author{J\'er\^ome Buzzi, Sylvain Crovisier, Todd Fisher
\footnote{S.C. was partially supported by the ERC project 692925 \emph{NUHGD}.  T.F.\ is supported by Simons Foundation grant \# 239708.}}

\makeatletter
\renewcommand{\l@section}{\@dottedtocline{2}{3.8em}{3.2em}}
\renewcommand{\l@subsection}{\@dottedtocline{3}{3.8em}{3.2em}}
\newcommand{\subsectionruninhead}{\@startsection{subsection}{2}{0mm}{-\baselineskip}{-0mm}{\bf\large}}
\newcommand{\subsubsectionruninhead}{\@startsection{subsubsection}{3}{0mm}{-\baselineskip}{-0mm}{\bf\normalsize}}
\makeatother

\begin{document}

\maketitle

\begin{abstract}
A number of techniques have been developed to perturb the dynamics of $C^1$-diffeomorphisms and to modify the properties of their periodic orbits. {For instance, one can} locally linearize the dynamics, change the tangent dynamics, or create local homoclinic orbits.  These techniques have been crucial for the understanding of $C^1$ dynamics, but {their most precise forms} have mostly been shown in the dissipative setting.  This work extends these results to volume-preserving and especially symplectic systems. {These tools underlie our study of the entropy of
$C^1$-diffeomorphisms in \cite{BCF}. We also give an application to the approximation of transitive invariant sets without genericity assumptions}.   
\end{abstract}

{\renewcommand{\thefootnote}{}%
}

\section{Introduction} 
According to often-cited words  of Poincar\'e,  periodic and heteroclinic orbits provide a ``breach into the fortress'' 
\cite[p.2]{Poincare1891} that is differentiable dynamics.  This key insight is still relevant today:
indeed, the closing and connecting lemmas established  \cite{pugh, Mane82, H, BC} in the $C^1$-topology lead to approximation by periodic orbits of chain-transitive sets and ergodic measures for $C^1$-generic systems
{(the  $C^1$-topology offers flexibility while preserving the differentiable structure).}
Some key dynamical properties, such as the existence or lack of a dominated splitting
on the tangent bundle, can then be detected on the periodic orbits.
{This text will focus on the perturbative approach.}

\paragraph{\rm\textit{The dissipative case.}}
The first results about linearization and modification of the tangent dynamics were motivated by  the $C^1$-stability conjecture:
Franks~\cite{franks} proved that one can perturb the tangent dynamics
over a given finite set by a small $C^1$-perturbation.
With this technique, the possible changes to the Lyapunov exponents and the angle between
stable and unstable spaces were first studied by 
Pliss~\cite{pliss}, Liao~\cite{liao}, and Ma\~n\'e~\cite{Mane82}.  
This very
local {analysis} has been {systematized} in~\cite{BGV} and more recently
in~\cite{BochiBonatti}.

The investigation of robust transitivity~\cite{DPU, BDP} led to a new approach where one is allowed to choose the periodic orbit supporting the perturbation.
Taking advantage of homoclinic orbits, a notion of ``transition"
(close to the specification property)
 was used to find  periodic points exhibiting the properties necessary {for} the perturbations. 
This second technique is powerful but requires that the initial system already has
homoclinic orbits. Moreover the perturbation is realized along a periodic orbit
which may approximate a large set. {Hence the perturbation is not local and its}  support is
difficult to control.

Perturbations with controlled support creating rich dynamics have been recently
built in~\cite{PS, wen,Gourmelon2010}.  These results {yield}  homoclinic tangencies and transverse homoclinic orbits in an arbitrarily small neighborhood of a hyperbolic periodic orbit.
Combining an improved
Franks lemma with perturbative results of linear systems, Gourmelon in~\cite{Gourmelon2010,Gourmelon2014} has shown how to perform various perturbations while preserving homoclinic relations.
This is crucial in many applications, e.g., in order to work inside a given homoclinic class, see \cite{BCDG,BCS,potrie} among others.

For a survey of the dissipative case, we refer to~\cite{C}.

\paragraph{\it The conservative case.}
For volume-preserving or symplectic systems,
this topic has not been so systematically investigated.
It has been shown that elliptic points of symplectic
diffeomorphisms characterize a lack of hyperbolicity~\cite{newhouse, arnaud, ArBC, HT, saghin-xia,CH} in dimensions 2 and 4.
We mention that some works {in smooth ergodic theory (dealing with the tangent dynamics over Lebesgue-almost every point instead of over periodic orbits) have developed related} arguments, see~\cite{bochi,BV}.
Other interesting properties pertaining to the entropy have been studied~\cite{CT, catalan1, catalan2, AACS}.
We note that these recent results mainly use the transition approach from~\cite{BDP} and do not provide local versions in the symplectic setting.

\paragraph{\it Results.}
Our goal is to systematically extend 
the perturbation tools to the conservative settings, trying as much as possible
to follow the local approach of~\cite{BochiBonatti}. Let us list our results
{(deferring the precise statements)}:
 \begin{itemize}\setlength\itemsep{0em}
  \item[--] Franks' lemma, linearization and preservation of homoclinic connections (Theorems \ref{t.linearize} and \ref{l.gourmelon}),
  \item[--] perturbation of the spectrum of periodic orbits: {achieving simplicity} 
(Proposition \ref{p.simple}), {realness (Proposition \ref{p.real}) or equal modulus for the stable and for the unstable eigenvalues} (Theorem \ref{t.BochiBonatti});
  \item[--] further perturbations of the tangent dynamics above periodic orbits:
   making the angle between stable and unstable spaces arbitrarily small (Theorem \ref{t.angle});
  \item[--] birth of homoclinic tangencies for a hyperbolic periodic orbit without strong dominated splitting (Theorem  \ref{t.gourmelon}).
 \end{itemize}

\paragraph{\it Consequences.} This paper started during the preparation of~\cite{BCF} which {studies}
the entropy of $C^1$-diffeomorphisms under a lack of domination. Hence, several applications
of the present paper are explained there. {We also note that the tools we present allow an easy and complete proof of such basic results as the necessity of a dominated splitting for robust transitivity (see the applications given in \cite{avila}).} We close this introduction by giving an
additional application.

An invariant compact set $\Lambda$ for a diffeomorphism $f$ on a boundaryless manifold $M$
has a \emph{dominated splitting} if there exists a decomposition
$TM|_{\Lambda}=E\oplus F$ of the tangent bundle of $M$ above $\Lambda$ in two invariant continuous subbundles and an integer $N\geq 1$ such that for all $x\in \Lambda$, all $n\geq N$ and all unit vectors $u\in E(x)$ and $v\in F(x)$ we have,
 $$
    \|Df^nu\|\leq \|Df^n v\|/2.
 $$

\begin{theorem}\label{t.horseshoes}
Let $\Lambda$ be a transitive invariant infinite compact set for a $C^1$-diffeo\-mor\-phism $f$
on a manifold $M$. If $\Lambda$ has no dominated splitting,
then there exists a diffeomorphism $g$ arbitrarily $C^1$-close to $f$
having a horseshoe $K$ that is arbitrarily close to $\Lambda$ for the Hausdorff topology.
If $f$ preserves a volume or a symplectic form, one can choose $g$ to preserve it also.
\end{theorem}

We point out that {this theorem makes} no genericity assumption on the diffeomorphism $f$.
In particular we do not suppose the existence of periodic orbits, so that the technique of transitions developed in~\cite{BDP}
cannot be used here.

\paragraph{\it Comments and questions.} 
We stress that the extension of the arguments of \cite{BochiBonatti} and others to the conservative and especially to the symplectic setting is not {direct}. 
For instance, we were led to modify Gourmelon's approach to Theorem \ref{t.gourmelon}. {The new argument is somewhat simpler, even in the dissipative setting.}

Our techniques sometimes provide slightly weaker results in the symplectic case than the ones available in the dissipative case. { For instance, consider} a periodic cocycle with a large, {given} period and without strong dominated splitting. Can a small perturbation make: (1)
All eigenvalues real? 
(2) All stable (all unstable) eigenvalues of equal moduli? 

{For (1), we need to assume the cocycle to be hyperbolic,} see Remark~\ref{r.difficulty1}. {For (2),} we need to go to a possibly unbounded multiple of the period, see Remark \ref{r.BochiBonatti}.
\medskip

We also point out that although the perturbations are only small {in the $C^1$-topology}, the resulting diffeomorphism often has the same regularity as the unperturbed system.
{Our results could thus} be used to provide examples and counter-examples with higher regularity in both the dissipative and conservative settings.
\medskip

\noindent{\it Acknowledgements.}
{We are grateful to Nicolas Gourmelon for discussions about these perturbation techniques.}

\section{Preliminaries}\label{s.prelim}

In this section we review some properties of smooth dynamics and state a few fundamental perturbation properties that will be used throughout.

\paragraph{Space of diffeomorphisms.}
Let $M$ be a compact connected boundaryless Riemannian manifold with dimension $d_0$.
The tangent bundle is endowed with a natural distance: considering the Levi-Civita connection,
the distance between $u\in T_xM$ and $v\in T_yM$ is the infimum of $\|u-\Gamma_\gamma v\|+\text{Length}(\gamma)$ (where $\Gamma_\gamma$ denotes the parallel transport) over $C^1$-curves $\gamma$ between $x$ and $y$.
Let $\Diff^1(M)$ denote the space of $C^1$-diffeomorphisms of $M$
and let $d_{C^1}$ be the following usual distance defining the $C^1$-topology:
$$d_{C^1}(f,g)=\sup_{v\in T^1M} \max\big( d(Df (v),Dg (v)), \; d(Df^{-1}(v),Dg^{-1}(v))\big).$$
We say that $g$ is an \emph{$\varepsilon$-perturbation} of $f$ when $d_{C^1}(g,f)<\varepsilon$.%

\paragraph{Hyperbolic periodic points and homoclinic relations.}
Let $f\in \mathrm{Diff}{^r}(M)$ and $p$ be a periodic point for $f$.
We denote by $\pi(p)$ the {(minimal)} period of $p$ and by $\cO(p)=\{f^i(p):p=0,\dots,\pi(p)-1\}$  its orbit.  A periodic point $p$ is \emph{hyperbolic} if $Df^{\pi(p)}(p)$ has no eigenvalues on the unit circle.  In this case there exists a stable manifold of $p$
denoted
$W^s(p)=\{ y\in M\, :\, d(f^{\pi(p)n}(y), p)\to 0, n\to \infty\}$. It is a $C^r$ immersed submanifold tangent to the stable eigenspace at $p$.  There is a similar definition for an unstable manifold.  So we have $T_pM=E^s\oplus E^u$ where $E^s$ is the stable eigenspace for $p$ and $E^u$ is the unstable eigenspace for $p$.  
The periodic point is a \emph{saddle} if the orbit $\cO(p)$ is hyperbolic and is neither a sink nor a source (so both $E^s$ and $E^u$ are nontrivial).

Two hyperbolic periodic orbits $O_1,O_2$ are \emph{homoclinically related} if 
\begin{itemize}
\item[--] $W^s(O_1)$ has a transverse intersection point with $W^u(O_2)$, and 
\item[--] $W^u(O_1)$ has a transverse intersection point with $W^s(O_2)$.
\end{itemize}
Equivalently, the orbits $O_1$ and $O_2$ are both contained in the same horseshoe (i.e., a topologically transitive, $0$-dimensional compact invariant subset which is hyperbolic and locally maximal).

\paragraph{Conservative diffeomorphisms.}
If $\omega$ is a volume or a symplectic form\footnote{We assume that $M$ and $\omega$ are $C^\infty$ smooth, see Remark 2 in \cite{avila}.} on $M$, one denotes by $\Diff^1_\omega(M)$ the subspace of diffeomorphisms which preserve $\omega$.
The charts $\chi \colon U\to \RR^{d_0}$ of $M$ that we will consider will always send $\omega$ on the standard Lebesgue volume
or symplectic form of $\RR^{d_0}$; it is well-known that any point admits a neighborhood with such a chart.

\paragraph{Symplectic linear algebra.}
For $d_0$ even, we write $d_0=2d$ and use the standard symplectic form, 
$$J:=\begin{pmatrix} 0 & I \\ -I & 0 \end{pmatrix}\quad \text{ (where $I=I_d$ is the $d\times d$-identity matrix),}$$
the Euclidean norm on $\RR^{2d}$ and the operator norm on matrices. For $d\times d$-matrices $\BA,\BB,\BC,\BD$,
 \begin{equation}\label{e.symplecticmatrix}
   \begin{pmatrix} \BA & \BB\\ \BC& \BD\end{pmatrix}\in Sp(2d,\RR) \iff
    \left\{\begin{array}{l} \BA^T\BC=\BA\BC^T,\\ \BB^T\BD=\BD^T\BB,\text{ and }\\
    \BA^T\BD-\BC^T\BB=I_d.\end{array}\right. 
 \end{equation}

The \emph{symplectic complement} $E^\omega$ of a subspace $E\subset \RR^{2d}$ is the linear subspace of vectors $v$ such that
$\omega(v,u)=0$ for all $u\in E$.

A subspace $E\subset \RR^{2d}$ is \emph{symplectic} if the restriction $\omega|E\times E$ of the symplectic form is non-degenerate (i.e. symplectic).
Two symplectic subspaces $E,E'$ are \emph{$\omega$-orthogonal} if for any $u\in E$, $u'\in E'$ one has $\omega(u,u')=0$.

A $d$-dimensional subspace $E$ is \emph{Lagrangian} if the restriction of the symplectic form vanishes, i.e. $E^\omega=E$.
Obviously, the stable (resp. unstable) space of a linear symplectic map is Lagrangian.

We now state and prove a few simple results we will need later.

\begin{lemma}\label{l.change-basis}
For any $\varepsilon>0$, there exists $\delta>0$ such that
if $(e_1,\dots,e_{d})$ is $\delta$-close to the first $d$ vectors of the standard basis $(E_1,\dots,E_{2d})$ of
$\RR^{2d}$ and generates a Lagrangian space,
then there exists $A\in Sp(2d,\RR)$ that is $\varepsilon$-close to the identity
such that $A(e_i)=E_i$ for $i=1,\dots,d$.
\end{lemma}
\begin{proof}
Let $M=(m_{ij})_{1\leq i\leq 2d,1\leq j\leq d}$ be the matrix defined such that  $e_j=\sum_{i=1}^{2d} m_{ij}E_i$ for $j=1,\dots,d$. Define a $2d\times 2d$ matrix as:
 $$
   \begin{pmatrix}\BA& 0\\ \BC&\BD\end{pmatrix} \text{ such that }\begin{pmatrix} \BA \\ \BC\end{pmatrix}=M
    \text{ and }\BD=(\BA^T)^{-1}.
  $$
This matrix sends $E_i$ to $e_i$ for $1\leq i\leq d$ and is symplectic since $(\BA^T\BC-\BC^T\BA)_{ij}=\omega(e_i,e_j)=0$. Hence its inverse has the claimed properties.
\end{proof}

\begin{corollary}\label{c.change-basis2}
There exists $C>0$ (only depending on $d$)
such that for any Lagrangian space $L\subset \RR^{2d}$,
there exists $A\in Sp({2d},\RR)$
satisfying
$$A(L)=\RR^{d}\times \{0\}^{d_0} \text{ and }
\|A\|,\|A^{-1}\|<C.$$
\end{corollary}
\begin{proof}
It is well-known that the symplectic group acts transitively on the Lagrangian spaces (by a variation of the preceding proof).
The claims follows from Lemma \ref{l.change-basis} and the compactness of the set of Lagrangian spaces.
\end{proof}

\paragraph{Perturbative tools.}  We will use the following definitions.

\begin{definition}\label{def-evNpert}
Consider $f\in\Diff^1(M)$, a finite set $X\subset M$, a neighborhood $V$ of $X$, and $\varepsilon>0$.
A diffeomorphism $g$ is an \emph{$(\varepsilon, V, X)$-perturbation} of $f$ if $d_{C^1}(f,g)<\varepsilon$ and $g(x)=f(x)$ for all $x$ outside of $V\setminus X$.
\end{definition}

\begin{definition}
Given $f\in \Diff^1(M)$, a periodic point $p$ of period $\ell =\pi(p)$ and $\varepsilon>0$,
an \emph{$\varepsilon$-path of linear perturbations at $\cO(p)$}
is a family of $\ell$ paths $(A_{i}(t))_{t\in[0,1]}$, $1\leq i\leq \ell$, of linear maps $A_i(t)\colon T_{f^i(p)}M\to T_{f^{i+1}(p)}M$
satisfying:
\begin{itemize}
\item[--] $A_{i}(0)=Df(f^i(p))$,
\item[--] $\sup_{t\in [0,1]} \big(\max{\|Df(f^i(p))-A_{i}(t)\|, \|Df^{-1}(f^i(p))-A^{-1}_{i}(t)\|}\big)<\varepsilon$.
\end{itemize}
\end{definition}
\medskip

The following folklore result modifies the image of one point.

\begin{proposition}\label{p.elementary}
For any $C,\varepsilon>0$,
there is $\eta>0$ with the following property.
For any $f\in \Diff^1(M)$ such that $Df$, $Df^{-1}$ are bounded by $C$,
for any pair of points $x,y$ such that $r:=d(x,y)$ is small enough,
one can find an $(\varepsilon, B(x,r/\eta),\emptyset)$-perturbation $g$ of $f$ satisfying $g(x)=f(y)$.
Furthermore, if $f$ preserves a volume or a symplectic form, one can choose $g$ to preserve it.
\end{proposition}

\paragraph{Periodic cocycles.}
Let $G$ be a subgroup of $GL(d_0,\RR)$. A \emph{periodic cocycle} with period $\ell$ is an integer $\ell\geq1$ together with a sequence $(A_i)_{i\in \ZZ}$ in $G$ such that $A_{i+\ell}=A_i$ for all $i\in\ZZ$.  The \emph{eigenvalues (at the period)} are the eigenvalues of $\cA:=A_{\ell}\dots A_{1}$. The cocycle is \emph{hyperbolic} if $\cA$ has no eigenvalue on the unit circle. The cocycle is \emph{bounded by $C>1$} if $\max (\|A_i\|,\|A_i^{-1}\|)\leq C$ for $1\leq i\leq \ell$. An \emph{$\varepsilon$-path of perturbations} is a family of periodic cocycles $(A_i(t))_{t\in [0,1]}$ in $G$ such that
$A_{i}(0)=A_i$ and
$\max{\|A_i-A_{i}(t)\|, \|A_i^{-1}-A^{-1}_{i}(t)\|}<\varepsilon$
for each $t\in [0,1]$.

\section{Franks' lemma}
We will use two strengthening of the classical Franks' lemma: we not only perturb the differential but also linearize on a neighborhood through a localized perturbation that keep the diffeomorphism conservative if it was so (Theorem \ref{t.linearize}) and even keep a homoclinic orbit if the periodic orbit stays hyperbolic (Theorem \ref{l.gourmelon}).

\subsection{Linearization}\label{s.linear}

\begin{theorem}[Franks' lemma with linearization]\label{t.linearize}
Consider $f\in \Diff^1(M)$, $\varepsilon>0$ small, a finite set $X\subset M$ and a chart $\chi\colon V\to \RR^{d_0}$ with $X\subset V$.
For $x\in X$, let $A_x\colon T_xM\to T_{f(x)}M$ be a linear map such that 
$$\max(\|A_x-Df(x)\|, \|A^{-1}_x-Df^{-1}(x)\|)<\varepsilon/2.$$
Then there exists an $(\varepsilon, V,X)$-perturbation $g$ of $f$ such that for each $x\in X$ the map $\chi\circ g \circ \chi^{-1}$ is linear in a neighborhood of $\chi(x)$ and
$Dg(x)=A_x$. Moreover if $f$ preserves a volume or a symplectic form, one can choose $g$ to preserve it also.
\end{theorem}
The dissipative case follows from a variation on Franks' original proof~\cite{franks}.
The symplectic case is obtained by a standard
argument involving generating functions (see for instance~\cite{Zehnder}).
The volume-preserving case requires an additional argument for which we will use two results from \cite{avila}.\footnote{
We note that a proof of Theorem~\ref{t.linearize} in the volume-preserving case has appeared \cite{teixera} while we were finishing this text.}

\begin{proof}[Proof in the volume-preserving case]
We assume that $V$ is an open subset of $\RR^{d_0}$ and suppose that $X=\{0\}\subset\RR^{d_0}$ and $Df(0)$ is the identity. We leave the reduction of the general case to this situation to the diligent reader.

\smallbreak\noindent\emph{Regularization.} For any open neighborhood $V'$ of $X$ with $\overline{V'}\subset V$, Theorem~7 of \cite{avila} (applied with $K_0=\emptyset$) gives a conservative, arbitrarily small $C^1$-perturbation with support in  $V$ and $C^\infty$ on $V'$. The image $f(0)$ may have changed but can be restored by a composing with a conservative perturbation of the identity (e.g., built using again Theorem 7 of \cite{avila}). Thus we can assume that $f$ is $C^\infty$ on $V'$.

\smallbreak\noindent\emph{Rectification.} 
We denote by $B_r$ the open ball with center $0\in\RR^{d_0}$ and radius $r>0$ and let 
 $$K:=(\overline{B_3}\setminus B_2)\cup \overline{B_1}\; \; \subset \; \;U:=(B_{3.1}\setminus\overline{B_{1.9}})\cup B_{1.1}.$$ 
 Let  $\eta:=\varepsilon/4(\|Df\|_\infty+1)$
where $\|\cdot\|_\infty$ is the supremum norm.  Corollary 5 of \cite{avila} yields:

\begin{lemma}\label{lemma-rectification}
There are an integer $k \geq 1$
and arbitrarily small number $\tau>0$ with the following property for $h\in C^\infty(B_{3.1},\RR^{d_0})$. 

If  $\|h-Id\|_{C^k}<\tau$ and $h|U$ is volume-preserving with {$vol(h(B_2\setminus \overline B_1))=vol(B_2\setminus \overline B_1)$}, then there is a volume-preserving map $\tilde h\in C^\infty(B_{3.1},\RR^{d_0})$ with $d_{C^1}(\tilde h,\Id)<\eta$
and $\tilde h=h$ on $K$.
\end{lemma}

\smallbreak\noindent\emph{Linearization.}
{Let us fix a map $a\in C^\infty(\RR,[0,1])$ with $a(r)=0$ if $r\leq 1.1$, $a(r)=1$ if $r\geq 1.9$.}
For $\rho>0$ small enough, $f_\rho(x):=\rho^{-1} f(\rho x)$ is well defined and arbitrarily $C^\infty$-close to the identity  on $B_{3.1}$. Define $h:B_{3.1}\to\RR^{d_0}$ by: 
  $$
       h(x) = (1-a(\|x\|))x+a(\|x\|)f_\rho(x).
  $$
 $h(B_2)=f_\rho(B_2)$, $h(B_1)=B_1$. Moreover, $\|h-Id\|_{C^k}\leq C(k)\|f_\rho-Id\|_{C^k}$ where $C(k)$ only depends on $k$ and $d_0$. Thus $\|h-Id\|_{C^k}<\tau$ for {$\rho>0$
small enough.}

Lemma \ref{lemma-rectification} yields a volume-preserving map $\tilde h:B_{3.1}\to\RR^{d_0}$ with $d_{C^1}(\tilde h,\Id)<\varepsilon/4$ that coincides with $f_\rho$ on $\overline{B_3}\setminus B_2$ and with the identity on $\overline{B_1}$. Replacing $f$ on $B_{3\rho}$ by {$\tilde h_{1/\rho}:=\rho\tilde h(./\rho)$}, we get a conservative $C^1$-diffeomorphism $\tilde f$ such that $\tilde f=Df(0)$ near $0$ and, for $\rho>0$ small enough,
 $$
    d_{C^1}(\tilde f,f) \leq d_{C^1}(\tilde h,\Id)+ d_{C^1}(\Id,f|B_{3.1\rho}) < \varepsilon/3.
 $$
This proves the theorem in the case $A_0=Df(0)=\Id$.

\smallbreak\noindent\emph{Perturbation of the differential.} We again use Lemma \ref{lemma-rectification}. As this procedure only effects small changes, one chooses a path $(A_t)_{0\leq t\leq 1}$  with 
$$\max(\|A_t-A_x\|, \|A_t^{-1}-A_x^{-1}\|)<2\varepsilon/3,\quad \text{ for any } t\in [0,1].$$ 
Choosing $N$ large enough,
one sets $A_i:=A_{i/N}$ for $i=0,\dots,N$, so that:
 $$
   \|A_i-A_{i+1}\|< \frac{\tau}{C(k)(\|Df^{-1}\|_\infty+1)}.
 $$
We are going to define $u_i\in\Diff^\infty(M)$ such that $d_{C^1}(u_i,f)<\varepsilon$ and $u_i|B_{3^{-i}\rho}=A_i$. We take $u_0=\tilde f$. For $0\leq i<N$, we build $u_{i+1}$ from $u_i$ by replacing it on $B_{3^{-i}\rho}$ by $A_{i+1}\circ\tilde h_{3^{-i+1}\rho}$ where $\tilde h$ given by the lemma starting from
 $$
    h(x):=(1-a(\|x\|))x+a(\|x\|)A_{i+1}^{-1}A_i(x)$$
    on $B_{3.1}.$
Indeed,
 $
   \|h-\Id\|_{C^k} \leq C(k)\|A_{i+1}^{-1}A_i-\Id\| \leq C(k)\|A_{i+1}^{-1}\|\cdot\|A_i-A_{i+1}\|< \tau.
 $
Thus, for $\tau>0$ small enough:
$$\begin{aligned}
   d_{C^1}(A_{i+1}\tilde h,A_{i}) &\leq \|A_{i+1}\|\cdot d_{C_1}(\tilde h,\Id) 
   + \|A_{i+1}\|\cdot\|A_{i+1}^{-1}\|\cdot \|A_{i+1}-A_i\|\\
   &
   < \frac{ (6/5)\|Df(0)\|\cdot \varepsilon}{4(\|Df\|_\infty+1)} +  \frac{ (36/25)\|Df(0)\|\cdot\|Df(0)^{-1}\| \cdot \tau }{C(k)(\|Df^{-1}\|_\infty+1)} <  \varepsilon/3
  \end{aligned}$$
so 
 $
   d_{C^1}(u_{i+1},f)\leq \max(d_{C^1}(u_i,\Id),\varepsilon/3+d_{C_1}(f|B_\rho,A_i))<\varepsilon,
   $ 
completing the induction.
Thus $g=u_N$ satisfies the claims of the theorem.
\end{proof}

\subsection{Homoclinic connections}\label{s.connection}

The next result further strengthens the linearizing version of Franks' lemma when the periodic orbit is kept hyperbolic: the perturbation preserves a given homoclinic relation. It has been
proved in~\cite{Gourmelon2014} in the dissipative case.

\begin{theorem}[Franks' lemma with homoclinic connection]\label{l.gourmelon}
Let  $f\in \Diff^1(M)$ and $\varepsilon > 0$ small.
Consider:
\begin{itemize}
\item[--] a hyperbolic periodic point $p$ of period $\ell =\pi(p)$,
\item[--] a chart $\chi\colon V\to \RR^{d_0}$ with $\cO(p)\subset V$,
\item[--] a hyperbolic periodic point $q$ homoclinically related to $\cO(p)$, and
\item[--] an $\varepsilon/2$-path of linear perturbations $(A_{i}(t))_{t\in[0,1]}$, $1\leq i\leq \ell$, at $\cO(p)$ such that the composition $A_{\ell}(t)\circ \dots A_{1}(t)$ is hyperbolic for each $t\in [0,1]$.
\end{itemize}
Then there exists an $(\varepsilon, V,\cO(p))$-perturbation $g$ of $f$ such that, for each $i$ the map
 $\chi\circ g\circ \chi^{-1}$ is linear and coincides with $A_{i}(1)$ near $f^i(p)$, and $\cO(p)$ is still homoclinically related to $q$.
\smallskip
Moreover if $f$ and the linear maps $A_i(t)$ preserve a volume or a symplectic form,
one can choose $g$ to preserve it also.
\end{theorem}
\begin{proof}
We need to extend \cite{Gourmelon2014} to the symplectic and volume-preserving cases, once the corresponding
versions of Franks' lemma (Theorem~\ref{t.linearize}) have been obtained.
One can assume that the orbits of $p$ and $q$ are distinct since otherwise the statement follows
from Theorem~\ref{t.linearize}.

In order to simplify the exposition,
we assume that $p$ is a fixed point, so that the linear perturbation
is reduced to a single path $A(t)=A_1(t)$.
Let $z$ be a transverse intersection point between $W^{s}_{loc}(p)$ and $W^u(q)$ and
 consider $N\geq 1$ large. We decompose the path $(A(t))_{t\in [0,1]}$ and consider the maps $A(k/N)$, $k=0,\dots,N$.
Since $q\notin\cO(p)$,
there exists a small ball $B_0$ centered at $p$
such that the backward orbit of $z$ does not intersect $B_0$.

Working in the chart and using inductively Theorem~\ref{t.linearize}, we build a first perturbation $h$
and a family of small nested balls $B_N\subset \dots \subset B_1\subset B_0$ centered at $p$
such that $f(B_{i+1})$ is contained in and much smaller than $B_i$, $h=f$ outside $B_0$, and such that
the differential $Dh$ is close to $A(k/N)$ on the region $B_{k-1}\setminus B_{k}$,
coincides with $A(k/N)$ on the boundary of
$B_k$, $1\leq k\leq N$, and with $A(1)$ inside $B_N$.

To keep the homoclinic connection,
we modify $h$ as follows. Let $f^n(z)$ be the first iterate of $z$ in $B_0$.
Since $A(1/N)$ and $Df(p)$ are $C^1$-close, the stable manifolds of $p$ for these two maps are $C^1$-close.
Using Proposition~\ref{p.elementary},
one can thus perturb $h$ at a point of $B_0\setminus B_1$ such that the forward orbit of $z$
under the new diffeomorphism meets the stable manifold of $p$ for $A(1/N)$ when it enters in $B_1$.
Since $A(1/N)$ and $A(2/N)$ are $C^1$-close, one can perturb $h$ at a point of $B_1\setminus B_2$ such that the forward orbit of $z$
under the new diffeomorphism meets the stable manifold of $p$ for $A(2/N)$ when it enters in $B_2$.
Perturbing inductively $N$ times on disjoint domains, {one} ensures that some {forward iterate of  $z$} belongs to the stable
manifold of $p$ for $A(1)$  in $B_N$. This {preserves a} transverse intersection at $z$ between
$W^{s}(p)$ and $W^u(q)$ for the new diffeomorphism. 

One similarly maintains a transverse intersection between
$W^{u}_{loc}(p)$ and $W^s(q)$ for the new diffeomorphism. Since the local stable and unstable manifolds of $p$ are separated,
this second perturbation can be chosen with disjoint support so as to preserve the transverse intersection $z$ between $W^s(p)$ and $W^u(p)$.
\end{proof}

\section{Perturbation of periodic linear cocycles}\label{ss.perturbative-cocycle}

In this section we show how to modify eigenvalues and create small angles between eigenspaces through perturbations.
The results in the dissipative case are essentially well-known.
{The perturbations preserve the Jacobian, hence the volume-preserving case also follows.}
The symplectic case, though, requires different arguments.

\subsection{Simple spectrum}\label{s.simple}

We first show how to obtain simple spectrum by arbitrarily small perturbations.

\begin{proposition}[Simple spectrum]\label{p.simple}
For any $d_0\geq 1$ and $\varepsilon>0$, any periodic cocycle in $GL(d_0,\RR)$, $SL(d_0,\RR)$ or $Sp(d_0,\RR)$ admits an $\varepsilon$-path of perturbations $(A_i(t))_{t\in [0,1]}$ with the same period $\ell$,
such that $\cA(t):=A_{\ell}(t)\dots A_{1}(t)$ satisfies:
\begin{itemize}
\item[] $\cA(1)$ has $d_0$ distinct eigenvalues; their arguments are in $\pi\QQ$.
\end{itemize}
In $GL(d_0,\RR)$ and $SL(d_0,\RR)$ one can furthermore require that the moduli of the eigenvalues of $\cA(t)$ are constant in $t\in[0,1]$.
\end{proposition}
\begin{proof}
We first sketch the proof in the dissipative setting (which is essentially contained in the claim of the proof of~\cite[Lemma 7.3]{BochiBonatti}).
It easily implies the volume-preserving case.
Since the necessary perturbation is arbitrarily small and can therefore be performed at a single iterate, it is enough to consider the case $\ell=1$. 

We proceed by induction on the sum $\delta$ of the dimensions of the eigenspaces corresponding to eigenvalues with multiplicity or arguments outside $\pi\QQ$. If $\delta=0$, there is nothing to show. Otherwise, we are in one of the following three cases in each of which, we can decrease $\delta$ by $2$. We note that the number of perturbations is at most $d_0/2$.

\smallbreak\noindent{\sl Case 1.} $\cA$ is $\RR$-conjugate to
$$\begin{pmatrix} \lambda I & B\\  0 & C \end{pmatrix}
\textrm{ where }I=\begin{pmatrix} 1 & 0\\  0 & 1 \end{pmatrix}.$$
One modifies $I$ as 
$$R_\theta=\begin{pmatrix} \cos(\pi\theta) & \sin(\pi\theta)\\  -\sin(\pi\theta) & \cos(\pi\theta) \end{pmatrix}.$$
For $\theta>0$ small, the eigenvalues of $R_\theta$ are non-real and of modulus $1$ and their arguments can be made rational.
The other eigenvalues are unchanged.

\smallbreak\noindent{\sl Case 2.} $\cA$ is $\RR$-conjugate to
$$\begin{pmatrix} \lambda N & B\\  0 & C \end{pmatrix}\textrm{ 
where }N=\begin{pmatrix} 1 & 1\\  0 & 1 \end{pmatrix}.$$
One modifies $N$ as 
$$N_\varepsilon=\begin{pmatrix} 1-\varepsilon & 1\\  -\varepsilon & 1 \end{pmatrix}$$
and one concludes as in Case 1.

\smallbreak\noindent{\sl Case 3.}  $\cA$ is $\RR$-conjugate to
$$\begin{pmatrix} rR_\theta & B\\  0 & C \end{pmatrix},$$ 
where $R_\theta$ is the rotation by angle $\pi\theta$.
One perturbs $\cA$ by modifying the angle $\theta$: during the perturbation the eigenvalues along the invariant $2$-plane $\RR^2\times \{0\}^{d_0-2}$
moves, but keep the same modulus. This concludes the proof in the dissipative case.

\bigbreak

\noindent{\sl We now investigate the symplectic case}
(See also~\cite[Section V]{robinson-global-analysis}
{and~\cite{robinson-periodique}}.)
Let us consider two integers $n,m\geq 1$ such that $2(n+m)=d_0$.
For each matrix $R\in GL(n,\RR)$ and each matrix
$\begin{pmatrix} A & B\\  C & D \end{pmatrix}$
in $Sp(2m,\RR)$, the matrices:
\begin{equation}\label{e.reduction}
\begin{pmatrix} R & RU & RS & RV\\
0 & A & X & B \\
0 & 0 & {}^tR^{-1} & 0 \\
0 & C & Z & D
 \end{pmatrix}
 \end{equation}
belong to $Sp(d_0,\RR)$,
once $S,U,V,X,Z,$ satisfy some relations independent from $R$.
One can thus perturb $R$ while keeping the matrix in $Sp(d_0,\RR)$.
\smallbreak

\smallbreak\noindent{\sl Case 1.} $\cA$ has a real eigenvalue $\lambda$.

One chooses an eigenvector $u$ and completes it as a symplectic basis.
This defines a bounded change of coordinates after which
$\cA$ takes the form~\eqref{e.reduction} above, with $n=1$
and $R=\lambda$. One can change the eigenvalues $\lambda,\lambda^{-1}$ 
by perturbing $R$, without affecting the other eigenvalues.

\smallbreak\noindent{\sl Case 2.}  $\cA$ has a complex eigenvalue $\sigma=\lambda e^{i\theta}$.

Thus there exists an invariant  $2$-plane $P$ where $\cA$ induces the map
$R:=\lambda {\rm Rot}_\theta$, where ${\rm Rot}_\theta$ is the rotation with angle $\theta$.
If the symplectic form vanishes on $P$, one completes it as before to get a symplectic basis and a bounded change of coordinates. Now $\cA$ takes
the form~\eqref{e.reduction} above.
One can perturb $R$ to get four distinct eigenvalues $\sigma, \bar \sigma, \sigma^{-1}, \bar \sigma^{-1}$,
all outside the unit circle \text{with arguments rational multiples of $\pi$.}
\smallbreak

In the case $P$ is symplectic, $\lambda$ equals $1$ and $R$ is a rotation.
The symplectic complement $P^\omega$ of $P$ is invariant
and $\cA$ is the Cartesian product of the symplectic rotation $R$ with some
map $Q\in Sp(d_0-2,\RR)$. One can again perturb $R$ to make the arguments of the eigenvalues $\sigma, \bar \sigma$, rational multiples of $\pi$ without modifying the rest of the spectrum.
\end{proof}

\subsection{Real eigenvalues}\label{s.real}

We now show that for sufficiently long cocycles one can make all the eigenvalues real.

\begin{proposition}[Real eigenvalues]\label{p.real}
For any $d_0\geq 1$, $C>1$, $\varepsilon>0$, there exists $T\geq 1$ with the following property.
Any periodic cocycle in $GL(d_0,\RR)$ or $SL(d_0,\RR)$ bounded by $C$ and with period $\ell\geq T$
admits an $\varepsilon$-path of perturbations $(A_i(t))_{t\in [0,1]}$ with the same period $\ell$ such that
\begin{itemize}
\item[--] the eigenvalues of $\cA(t):=A_{\ell}(t)\dots A_{1}(t)$ have moduli constant in $t\in[0,1]$,
\item[--] $\cA(1)$ has only real eigenvalues.
\end{itemize}
The same result holds in $Sp(d_0,\RR)$ if the cocycle is hyperbolic.
\end{proposition}

A previous result for surfaces was obtained in~\cite[Lemma 6.6]{BC}.
The dissipative case was proved in~\cite[Proposition 4.3]{BochiBonatti}
and the volume-preserving case is an immediate consequence.  We extend it to the symplectic case.

{\begin{remark}\label{r.difficulty1} We don't know if the same result for arbitrary symplectic cocycles
(in particular for cocycles admitting an invariant symplectic $2$-plane,
when $d_0\geq 4$).
\end{remark}}

\begin{proof}[Proof in the symplectic case]
By Corollary~\ref{c.change-basis2}, there exists a change of coordinates by a bounded cocycle such that
each map $A_i$ preserves the space $\RR^{d}\times \{0\}^{d}$ and the moduli of the eigenvalues of
$\cA$ along this space are smaller than $1$. Consequently, one can assume each map $A_i$ to be
a matrix of the form
\begin{equation}\label{e.triangle}
\begin{pmatrix} B_i^T & C_i\\ 0 & B_i^{-1} \end{pmatrix}.
\end{equation}
The dissipative version of Proposition~\ref{p.real} applied to the cocycle $(B_i)$ yields a path of perturbations $(U_i(t))_{t\in [0,1]}$ in $GL(d,\RR)$ such that the moduli of the eigenvalues
of the cocycles $(B_iU_i(t))$ are constant in $t$ and the eigenvalues of the composition
$(B_\ell U_\ell(1))\dots(B_1U_1(1))$ are all real.
To build the $\varepsilon$-path of perturbations, one can set  $t\mapsto A_i(t):=A_i. D_i(t)$,
where $D_i(t)$ is given by the symplectic matrix 
$$D_i(t):=\begin{pmatrix} U(t)^T & 0\\ 0 & U(t)^{-1} \end{pmatrix}.$$
\end{proof}

\begin{corollary}[Real simple eigenvalues]\label{c.real}
For any $d_0\geq 1$, $C>1$, $\varepsilon>0$, there exists $T\geq 1$ with the following property.
Any periodic cocycle in $GL(d_0,\RR)$ or $SL(d_0,\RR)$ bounded by $C$ and with period $\ell\geq T$
admits an $\varepsilon$-path of perturbations $(A_i(t))_{t\in [0,1]}$ with the same period such that
\begin{itemize}
\item[--] $\cA(t):=A_{\ell}(t)\dots A_{1}(t)$ is hyperbolic for any {$t>0$},
\item[--] $\cA(1)$ has $d_0$ distinct real eigenvalues
with modulus in $(0,1/2)\cup (2,+\infty)$.
\end{itemize}
The same result holds in $Sp(d_0,\RR)$ if the cocycle is hyperbolic.
\end{corollary}
 \begin{proof}
In the dissipative or volume-preserving cases, once all the eigenvalues are real, one can conjugate by a bounded cocycle (using Gram-Schmidt orthonormalization) to reduce
 to the case where all the $A_i$ are defined by triangular matrices. The result then follows easily by
 perturbing the diagonal coefficients.

 In the symplectic case, the proof is the same:
conjugacy by a bounded cocycle brings the cocycle $\cA$ to the form~\eqref{e.triangle}, where
 $B_i$ is a lower triangular matrix.
 \end{proof}

\subsection{Small angle}\label{s.angle}

We use the lack of $N$-dominated splitting to find a perturbation making the angle between the
stable and unstable bundles small.

\begin{theorem}[Small angle]\label{t.angle}
For any $d_0\geq 1$, $C>1$, $\varepsilon>0$, there exist $T,N \geq 1$ with the following property.
For any hyperbolic periodic cocycle $(A_i)_{i\in \ZZ}$ in $GL(d_0,\RR)$, $SL(d_0,\RR)$ or $Sp(d_0,\RR)$ bounded by $C$, with period $\ell\geq T$ and such that {the splitting defined by the stable and unstable eigenvalues is not an $N$-dominated splitting}, there exists an $\varepsilon$-path of perturbations $(A_i(t))_{t\in [0,1]}$ with period $\ell$, such that:
\begin{itemize}
\item[--]  $\cA(t):=A_\ell(t)\dots A_1(t)$ is hyperbolic for each $t\in[0,1]$,
\item[--] for some $j\in\{1,..., \ell\}$, the angle between the stable and unstable
spaces $E^s_{j}, E^u_{j}$ of the cocycle $(A_i(1))$ is smaller than $\varepsilon$.
\end{itemize}
\end{theorem}

This has been obtained in the dissipative (and volume-preserving) case in~\cite[Proposition 4.7]{Gourmelon2014}.  Previous results were obtained in~\cite{PS, wen}.

\begin{proof}[Proof of the symplectic case]
Applying Corollary~\ref{c.real},
(and assuming that the period is large enough),
one can perform a first $\varepsilon/2$-perturbation
and reduce to the case the cocycle has $d_0$ real eigenvalues with distinct moduli.
One can assume that the angle between $E^s,E^u$ is larger than $\varepsilon$ since
otherwise the {theorem} holds trivially.

Applying Lemma~\ref{l.change-basis}, one can find a bounded symplectic change of coordinates such that
each bundle $\RR^k\times \{0\}^{d_0-k+1}$, $1\leq k\leq d$, is invariant and contained in
the stable bundle (where as before we let $d_0=2d$).
The unstable bundle has the form $E^u_i=\{(L_i(u),u),\; u\in \RR^d\}$.
Since $E^u_i$ is Lagrangian, the matrix 
$$\Delta_i:=\begin{pmatrix} I_d & L_i\\ 0 & I_d\end{pmatrix}$$
is symplectic. Moreover since the angle between $E^s$ and $E^u$ is bounded away from zero,
the matrices $\Delta_i$ and $\Delta_i^{-1}$ are uniformly bounded.
After conjugating by the cocycle $(\Delta_i)$, we are thus reduced to the case the cocycle has
the form 
$$\begin{pmatrix} B_i^T & 0\\ 0 & B_i^{-1} \end{pmatrix}$$
where $B_i^T$ has been made upper triangular using as before Gram-Schmidt orthonormalization and has diagonal coefficients
$b_i(1),\dots,b_i(d)$. 
Consequently $B_i^{-1}$ is lower triangular and has
diagonal coefficients 
$$b_i(d+1)=b_i(1)^{-1},\dots,b_i(d_0)=b_i(d)^{-1}.$$
These coefficients define $d_0$ real cocycles.
\medskip

\paragraph{\it The case $d_0=2$.}
The two-dimensional case now follows from the established argument in the dissipative case. We will use the following more precise statement for our proof in higher dimensions.

\begin{lemma}\label{l.angle-2D}
For any $C>1$, $\varepsilon'>0$, there exist $T,N'\geq 1$ with the following property.
For any hyperbolic periodic cocycle $(D_i)$ of diagonal matrices in $SL(2,\RR)$ bounded by $C$,
with period $\ell\geq T$ and no $N'$-dominated splitting, there exists an $\varepsilon'$-path of
perturbations $(D_i(t))$ such that:
\begin{itemize}
\item[--] the eigenvalues of $\mathcal{D}(t):=D_\ell(t)\cdots D_1(t)$ have moduli constant in $t$;
\item[--] the {matrices} $U_i(t)=D_i^{-1}\cdot D_i(t)$ are upper triangular;
\item[--] for some $j\in \{1,..., \ell\}$, the angle between the stable and the unstable spaces
$E^s_{j},E^u_{j}$ of the cocycle $(D_i(1))$ is smaller than $\varepsilon'$.
\end{itemize}
\end{lemma}
\begin{proof}[Comment on the proof]
The proof follows the classical argument by Ma\~n\'e, see~\cite[Lemma 7.10]{BDV}.
Each matrix $U_i(t)$ either has one of the following forms.
\begin{itemize}
\item[--] $\begin{pmatrix} 1 & t\eta_i\\ 0 & 1 \end{pmatrix}$: it twists one of the bundles; or
\item[--] $\begin{pmatrix} (1+\eta_i)^t & 0\\ 0 & (1+\eta_i)^{-t} \end{pmatrix}$: it accentuates the contraction
or the expansion along the invariant bundles.
\end{itemize}
Consequently the first bundle $\RR\times \{0\}$ is invariant.
One can choose $\prod_{i=1}^\ell(1+\eta_i)=1$ so that the moduli of the
eigenvalues are unchanged.
\end{proof}

\paragraph{\it Choice of a symplectic plane.}
Let us come back to the general case.
One chooses $\varepsilon'>0$ small (see the condition later)
and takes $N'$ as given by the previous lemma.
We now explain how to fix the integer $N$ and how
under the assumptions of Theorem~\ref{t.angle}, one can find
$1\leq r \leq d$ such that
$(b_i(r))_{1\leq i\leq\ell}$ is not $N'$-dominated by the cocycle $(b_i(d+r))_{1\leq i\leq\ell}$.

\begin{lemma}
For any $N'\geq 1$, there exists $N$ with the following property.
If for any $j\in \{1,\dots,d\}$ and any $k\in \{d+1,\dots, d_0\}$
the cocycle $(b_i(j))_{1\leq i\leq\ell}$ is $N'$-dominated by the cocycle $(b_i(k))_{1\leq i\leq\ell}$,
then the bundle $E^s$ is $N$-dominated by $E^u$ for the cocycle $(A_i)_{1\leq i\leq\ell}$.
\end{lemma}
\begin{proof}
Since the cocycle $(B_i^T)$ is upper triangular and uniformly bounded, there exists a uniform
constant $K>0$ such that for any $n\geq 1$,
the norm of
$B_{i+n}^T\cdots B_{i+1}^T$ is bounded by
$K\max_{1\leq r\leq d} |b_{i+n}(r)\cdots b_{i+1}(r)|$.
Similarly, the co-norm (i.e. the minimal norm of the image of a unit vector)
of $B_{i+n}^{-1}\cdots B_{i+1}^{-1}$ is bounded from below by
$K^{-1}\min_{d+1\leq r\leq d_0} |b_{i+n}(r)\cdots b_{i+1}(r)|$.

Consequently, $ E^s$ is $N$-dominated by $E^u$
provided for all $n\geq N$
$$K^2\max_{1\leq r\leq d} |b_{i+n}(r)\cdots b_{i+1}(r)|\leq \frac 1 2\min_{d+1\leq r\leq d_0} |b_{i+n}(r)\cdots b_{i+1}(r)|.$$
One thus chooses $m\geq 1$ such that $K^2<2^{m-1}$ and sets $N'=mN$.
\end{proof}

Let us assume that $(b_i(j))$ is not $N'$-dominated by the cocycle $(b_i(k))$
for some $1\leq j\leq d<k\leq d_0$.
There exists $i\in \ZZ$ and $n\geq N'$ such that
$$|b_i(j)...b_{i+n-1}(j)| > 1/2 |b_i(k)...b_{i+n-1}(k)|.$$
Since $b_i(d+s)=b_i(s)^{-1}$ for each $1\leq s\leq d$, one also gets
$$|b_i(k-d)...b_{i+n-1}(k-d)| > 1/2 |b_i(j+d)...b_{i+n-1}(j+d)|.$$
Let us assume that $(b_i(k-d))$ is $N'$-dominated by $(b_i(k))$.
This gives
$$|b_i(k)...b_{i+n-1}(k)| \geq  2 |b_i(k-d)...b_{i+n-1}(k-d)|.$$
Combining these inequalities, one deduces
$$|b_i(j)...b_{i+n-1}(j)|>1/2 |b_i(j+d)...b_{i+n-1}(j+d)|,$$
and $(b_i(j))$ is not $N'$-dominated by $(b_i(j+d))$.
We have thus shown:

\begin{lemma}\label{lem-nodom-2}
If there are $1\leq j\leq d<k\leq d_0$ such that
the cocycle $(b_i(j))_{1\leq i\leq\ell}$ is not $N'$-dominated by the cocycle $(b_i(k))_{1\leq i\leq\ell}$,
then there exists $1\leq r\leq d$ such that
$(b_i(r))_{1\leq i\leq\ell}$ is not $N'$-dominated by the cocycle $(b_i(d+r))_{1\leq i\leq\ell}$.
\end{lemma}

\paragraph{\it The perturbation.}
Let us assume now that $(b_i(r))_{1\leq i\leq\ell}$ is not $N'$-dominated by $(b_i(d+r))_{1\leq i\leq\ell}$.
One can decompose the stable and unstable bundles as follows
$$\RR^d\times \{0\}^d=\RR^{r-1}\times \RR \times\RR^{d-r}\times \{0\}=E^s_1\oplus E_2^s\oplus E^s_3,$$
$$\{0\}^d\times \RR^d=\{0\}^d\times\RR^{r-1}\times \RR \times\RR^{d-r}=
E^u_1\oplus E_2^u\oplus E_3^u.$$
The spaces $E_*:=E^s_*\oplus E^u_*$ for $*=1, 2$ or $3$ are symplectic.
In the coordinates $E_1\oplus E_2\oplus E_3$,  the cocycle takes the form
$$A_i=\begin{pmatrix} D_{1,i} & D_{2,i} & D_{3,i}\\ 0 & D_{5,i} & D_{6,i}\\
0 & 0 & D_{9,i}\end{pmatrix}.$$
Each matrix 
$$D_{5,i}=\begin{pmatrix} b_i(r) & 0\\ 0 & b_i(d+r)\end{pmatrix}$$
is two-dimensional, symplectic, and diagonal; by our assumptions it has no
$N'$-dominated splitting
and the composition
$D_{5,\ell}\dots D_{5,1}$ is hyperbolic.
The matrices $D_2$ have the form 
$$D_{2,i}=\begin{pmatrix} V_i & 0\\ 0 & V'_i\end{pmatrix},$$
where $V_i$ and $V'_i$ are $1\times (r-1)$ matrices.

Let us apply Lemma~\ref{l.angle-2D} to the cocycle $(D_{5,i})$:
this gives an $\varepsilon'$-path of perturbations $(D_{5,i}(t))$.
We set $U_i(t)=D_{5,i}^{-1}\cdot D_{5,1}(t)$ and define the symplectic matrices
$$A_i(t)=A\cdot \begin{pmatrix} I_{r-1} &  0& 0\\ 0 & U_i(t) & 0\\
0 & 0 & I_{d-r}\end{pmatrix}=\begin{pmatrix} D_1 & D_2\cdot U_i(t) & D_3\\ 0 & D_{5,i}(t) & D_6\\
0 & 0 & D_9\end{pmatrix}.$$
If $\varepsilon'$ has been chosen small enough, this is a $\varepsilon$-path of perturbations.
By construction, the modulus of the eigenvalues are unchanged, hence the cocycle
is hyperbolic for each $t\in [0,1]$.

By construction, there exists some $1\leq j\leq\ell$ such that the stable and unstable spaces at index $j$ of the cocycle $(D_{5,i}(1))$ contain nonzero vectors $w^s,w^u\in E^s_2\oplus E^u_2$ which satisfy $\angle(w^s,w^u)\leq\varepsilon'$.

Lemma \ref{l.angle-2D} keeps $E^s_1\oplus E^s_2$ and therefore $w^s$ in the stable space of $(A_i(1))$. As $E_1\oplus E_2$ remains invariant for $(A_i(1))$,  the unstable space of $(D_{5,i}(1))$ lifts inside $E^1\oplus E^2$ to that of  $(A_i(1))$. Thus we get an unstable vector for $(A_i(1))$
of the form $w^u+v^s_1+v^u_1$ with $v^*_1\in E^*_1$.

As $E^u_1$ is contained in the unstable space of $(A_i(1))$,
the vector $\bar w^u:=(w^u+v^s_1+v^u_1)-v^u_1$ is in the unstable space for $(A_i(1))$. 
As $E^s_1$ is an invariant stable subspace for $(A_i(1))$, the vector $\bar w^s:=w^s+v^s_1$ is in the stable space. 

To conclude the proof of the theorem, observe that $\bar w^u-\bar w^s=w^u-w^s$ and $\|\bar w^*\|\geq\|w^*\|$ for $*=u,s$ so $\angle(\bar w^u,\bar w^s)\leq\angle(w^u,w^s)\leq\varepsilon'$. 
\end{proof}

\subsection{Mixing the exponents}\label{s.mixing}

The lack of strong domination leads to further  perturbations making all  
stable (resp. unstable) eigenvalues to have equal modulus. 
The next statement has been proved in the dissipative (and volume-preserving) setting in~\cite[Theorem 4.1]{BochiBonatti}.

\begin{theorem}[Mixing the exponents]\label{t.BochiBonatti}
For any $d_0\geq 1$, $C>1$, $\varepsilon>0$, there exists $N \geq 1$ with the following property.
For any hyperbolic periodic cocycle $(A_i)_{i\in \ZZ}$ in $GL(d_0,\RR)$, $SL(d_0,\RR)$ or $Sp(d_0,\RR)$ bounded by $C$, with period $\pi$
and no $N$-dominated splitting, there exists an $\varepsilon$-path of perturbations $(A_i(t))_{t\in [0,1]}$ with period $\ell$, a multiple of $\pi$, such that:
\begin{itemize}
\item[--] $\cA(t):=A_\ell(t)\dots A_1(t)$ is hyperbolic for any $t\in[0,1]$,
\item[--] $t\mapsto |\det(\cA(t)_{|E^s})|$ and $t\mapsto |\det(\cA(t)_{|E^u})|$ are constant in $t\in[0,1]$,
\item[--] the stable (resp. unstable) eigenvalues of $\cA(1)$ have the same moduli.
\end{itemize}
\end{theorem}

\begin{remark}\label{r.BochiBonatti}
In fact, a stronger statement is proven in \cite{BochiBonatti} for $GL(d_0,\RR)$ and $SL(d_0,\RR)$.  Namely,
there exists $T\geq 1$ which only depends on $d_0,C,\varepsilon,$ and  $N$, such that
any multiple $\ell\geq T$ of $\pi$ satisfies the conclusion of Theorem~\ref{t.BochiBonatti}.
In particular, if the period $\pi$ is larger than $\ell$, one can choose $\ell=\pi$.
We do not know if this uniformity holds in $Sp(d_0,\RR)$.
{Note also that \cite{BochiBonatti} allows to realize other spectra for non-dominated
cocycles in $GL(d_0,\RR)$. We do not know to what extend this generalizes to
the symplectic case.}
\end{remark}

It remains to prove the symplectic case of Theorem \ref{t.BochiBonatti}. The proof is by reduction to the dissipative case.
We will see that, in the symplectic category,  any hyperbolic cocycle without strong dominated splitting admits a perturbation (maybe with a larger period) whose restriction to its stable subbundle is also without any strong dominated splitting. The dissipative case of the theorem can then be applied 

Let us fix numbers $C,\varepsilon>0$. In this section we always consider hyperbolic cocycles $A=(A_i)_{i\in\ZZ}$ in $Sp(d_0,\RR)$ bounded by $2C$. Choosing $\varepsilon>0$ small enough, any $\varepsilon$-perturbation of a cocycle bounded by $C$ is still bounded by $2C$.
\medskip

\begin{lemma}\label{c.simple-case}
There exists $N_0\geq 1$ such that if the stable bundle of $A=(A_i)$ has no $N_0$-dominated splitting, then
there exists an $\varepsilon/2$-path of perturbations of $A$ (possibly with larger period)
satisfying the conclusion of Theorem~\ref{t.BochiBonatti} in $Sp(d_0,\RR)$.
\end{lemma}

\begin{proof}
As before, one uses Corollary~\ref{c.change-basis2} to reduce to the case of cocycles of the form 
$$\begin{pmatrix} B_i^T & C_i\\ 0 & B_i^{-1} \end{pmatrix},$$ where $\RR^{d}\times \{0\}^{d}$
is the stable space. 
Since the cocycle $(B_i)$ has no strong dominated splitting, the version of Theorem~\ref{t.BochiBonatti} for $GL(d,\RR)$ provides  a path of perturbations $(B_i(t))$ (with possibly larger period $\ell$),
 whose composition
$\cB(t)=B_\ell(t)\dots B_1(t)$ has constant Jacobian, only eigenvalues with modulus smaller than $1$,
and such that $\cB(1)$ has all its eigenvalues with the same modulus.
One concludes as in Section \ref{s.real}.
\end{proof}

The key to this reduction is the next proposition
{which analyzes the dominated decompositions
$E^s=E^{ss}\oplus E^{cs}$ of the stable spaces.
The domination is not quantified and only requires that the eigenvalues
along $E^{ss}$ have smaller moduli than along $E^{cs}$.}

\begin{proposition}\label{p.BochiBonatti-reduced}
Given $N_0$, $\varepsilon'>0$, $C>1$, and $j\in \{1,\dots,d-1\}$, there exists an integer $N'=N'(N_0,j,\varepsilon', C)\geq 1$
 with the following property. Let  $A=(A_i)$ be any hyperbolic cocycle in $Sp(d_0,\RR)$ bounded by $2C$
 with period $\pi$ such that its composition $A_\pi\dots A_1$ has $d_0$
 eigenvalues with pairwise distinct moduli. If {the stable space of} $A$ admits a splitting $E^s=E^{ss}\oplus E^{cs}$ with $\dim(E^{ss})=j$, then
\begin{itemize}
\item[--] either $A$ admits a $N'$-dominated splitting,
\item[--] or there exists an  $\varepsilon'$-path of perturbations $A(t)=(A_i(t))$ of $A$ with possibly larger period $\ell$ such that
{\begin{enumerate}
\item\label{i.1} each $\cA(t):=A_\ell(t)\dots A_1(t)$ is hyperbolic,
\item\label{i.3} $A(t)$ preserves $E^{ss}$
and coincides with $A$ on $E^{ss}$, for all $0\leq t\leq 1$,
\item\label{i.4} $\cA(t)$ has
a dominated splitting $E^s(t)=E^{ss}\oplus E^{cs}(t)$ for all $0\leq t\leq 1$,
\item\label{i.2} $t\mapsto |\det(\cA(t)_{|E^s})|$ and $t\mapsto |\det(\cA(t)_{|E^u})|$ are constant in $t$,
\item\label{i.5} the splitting {$E^s(1)=E^{ss}\oplus E^{cs}(1)$} for $A(1)$ is not $N_0$-dominated.
\end{enumerate}}
\end{itemize}
\end{proposition}
\noindent

\begin{proof}[Proof of Theorem~\ref{t.BochiBonatti} from Proposition~\ref{p.BochiBonatti-reduced}]

Lemma \ref{c.simple-case} gives an integer $N_0$, depending on $\varepsilon$. We will apply the proposition for each possible $E^{ss}$-dimension $1\leq j\leq d-1$ to remove any $N_0$-dominated splitting in the stable subbundle by an $\varepsilon/2$-perturbation. The $\varepsilon/2$-perturbation given by Lemma \ref{c.simple-case} will finish the proof. More precisely, {we pick $0<\varepsilon_{d-1}<\varepsilon/2d$ and $N_{d-1}\geq N'(N_0,d-1,\varepsilon_{d-1},C)$ and then, inductively, select $N_j,\varepsilon_j$ for $1\leq j<d-1$, given $N_{j+1}$, by:}
 \begin{itemize}
     \item[--] $0<\varepsilon_{j}<\varepsilon/2d$ so small that non $2N_{j+1}$-dominated splitting for some cocycle implies non $N_{j+1}$-dominated splitting for any $\varepsilon_{j}$-perturbation; and
     \item[--] $N_{j}\geq N'(N_0,j,\varepsilon_{j},C)$ with $N_j\geq 2N_{j+1}$. 
 \end{itemize}
Now, given a cocycle $A$ without $N_1$-dominated splitting, we inductively get cocycles $A^{(1)},\dots,A^{(d)}$ by setting $A^{(1)}:=A$ and taking, for $2\leq j\leq d$, $A^{(j)}$ to be an $\varepsilon_{j-1}$-perturbation of $A^{(j-1)}$ satisfying:
 \begin{itemize}
     \item [] $A^{(j)}$ has no $N_0$-dominated splitting of index strictly less than $j$ inside $E^s$ and, if $j<d$, no $N_j$-dominated splitting in $\RR^{2d}$.
 \end{itemize}
In particular, the distance from $A$ to $A^{(d)}$ is less than $\varepsilon_1+\dots+\varepsilon_{d-1}<\varepsilon/2$ and the restriction $A^{(d)}|E^s$ has no $N_0$-dominated splitting. Theorem~\ref{t.BochiBonatti} now follows from Lemma~\ref{c.simple-case}.
\end{proof}

It remains to prove Proposition \ref{p.BochiBonatti-reduced}.
Note that the symmetry {of the spectrum of a symplectic cocycle}
implies that there also exists a {dominated} splitting $E^u=E^{cu}\oplus E^{uu}$ with $\dim E^{uu}=\dim E^{ss}$.
We prove three preliminary lemmas.

\begin{lemma}\label{c.rectify}
For any $\varepsilon'>0$ and any integer$N_0$, there exists $\eta>0$ with the following property.
If $(A_i)$ has a {dominated} splitting $E^{ss}\oplus E^{cs}\oplus E^{cu}\oplus E^{uu}$
with $\dim E^{ss}=j$ such that $E^s=E^{ss}\oplus E^{cs}$ is
$N_0$-dominated,
then for any $k\in\ZZ$ and any space $\widehat E^{cs}_{k+1}\subset E^{cs}_{k+1}\oplus E^{cu}_{k+1}$
that is $\eta$-close to $E^{cs}_{k+1}$ and with equal dimension,
there exists an $\varepsilon'$-path of perturbations
$(A_i(t))$ in $Sp(2d,\RR)$ such that
\begin{itemize}
\item[--] $A_{k}(1). E^{cs}_{k}=\widehat E^{cs}_{k+1}$, and
\item[--] the maps $A_i(t)$ restricted to $E^{ss}$ are constant in $t$.
\end{itemize}
\end{lemma}

\begin{proof}
Since $E^s=E^{ss}\oplus E^{cs}$ is $N_0$-dominated,
the angle between $E^{ss}$ and $E^{cs}$ is uniform.
Arguing as in Corollary~\ref{c.change-basis2},
one finds $C'>0$, depending only on $C,d,N_0$,
such that, after conjugacy by a cocycle in $Sp(2d,\RR)$ and bounded by $C'$, the subspace $E^{ss}$ becomes $\RR^j\times\{0\}^{2d-j}$
and the subspace $E^{cs}$ becomes $\{0\}^j\times \RR^{d-j}\times \{0\}^d$. 
In these new coordinates, there exists a linear map
$B\colon \RR^d\to \RR^d$ such that $\RR^j\times \{0\}^{d-j}\subset \ker(B)$ and
$\widehat E^{cs}_k=\{(u,B(u)):\; u\in E^{cs}_k\}$.
Note that $B$ is (uniformly) close to $0$ when $\eta$ is small.
One builds the path of perturbations by composing 
$A_{k-1}$ with the symplectic matrices, given (in the new coordinates) by:
$$
  U(t)=\begin{pmatrix} I & 0\\ tB & I \end{pmatrix}
  .$$
\end{proof}

We denote the partial compositions $A_{i+n-1}\cdots A_i$ by {$A_i^n$}.

\begin{lemma}\label{l.break-domination}
For any $\eta>0$ and any integer $N_0$, there is $\widetilde N$ with the following property. Let
$(A_i)$ be a cocycle with a {dominated} splitting $E^{ss}\oplus E^{cs}\oplus E^{cu}\oplus E^{uu}$. If $E^s=E^{ss}\oplus E^{cs}$ is $N_0$-dominated and
$E^{ss}\oplus (E^{cs}\oplus E^{cu})$ is not $\widetilde N$-dominated,
then there exist integers $i_0\in \ZZ$, $m_0\geq N_0$ and unit vectors $u^c\in E^{cs}_{i_0}\oplus E^{cu}_{i_0}$,
$u^{ss}\in E^{ss}_{i_0}$ such that
\begin{enumerate}[(a)]
\item the line $A^{m_0}_{i_0}(\RR .u^c)$ is $\eta$-close to $E^{cs}_{i_0+m_0}$, and
\item
$\|A^{N_0}_{i_0}(u^{ss})\|>\frac 1 2 {\|A^{N_0}_{i_0}(u^{c})\|}$.
\end{enumerate}
\end{lemma}

\begin{proof}
Let $a=C^{-2N_0}/10$ and let $m_0$ be an integer larger than
$N_0|\log(\eta a)|$.
Since $E^{ss}\oplus E^{cs}$ is $N_0$-dominated,
for any $i\in \ZZ$ and any unit vectors
$u^{ss}\in E^{ss}_i$ and $u^{cs}\in E^{cs}_i$ we have
\begin{equation}\label{e.eq1}
  \|A^{m_0}_{i}(u^{cs})\|> \frac{3}{2\eta a} \|A^{m_0}_{i}(u^{ss})\|.
\end{equation}
Pick an integer $\widetilde N>\max(N_0,2m_0^2(\log C+2)/\log\tfrac32)$.
The lack of $\widetilde N$-dominated splitting of $E^{ss}\oplus (E^{cs}\oplus E^{cu})$ yields $j_0\in \ZZ$, $n\geq \widetilde N$
and some unit vectors
$u^{ss}_0\in E^{ss}_{j_0}$ and $u^c_0\in E^{cs}_{j_0}\oplus E^{cu}_{j_0}$ such that
 $$
   \|A^{n}_{j_0}(u^{ss}_0)\|\geq \tfrac12 \|A^{n}_{j_0}(u^{c}_0)\|.
 $$
Thus the positive numbers
 $$
    a_i:=\frac{\|A^{i-j_0+1}_{j_0}(u^{ss}_0)\|}{\|A^{i-j_0}_{j_0}(u^{ss}_0)\|}
    \times
        \frac{\|A^{i-j_0}_{j_0}(u^{c}_0)\|}{\|A^{i-j_0+1}_{j_0}(u^{c}_0)\|}
        \leq C^2
        \quad (i=j_0,\dots,j_0+n-1)
 $$
satisfy $a_{j_0}\dots a_{j_0+n-1}>1/2$. From $n\geq \tilde N$ and the choice of $\widetilde N$,
one deduces that there is an integer $i_0$ with $j_0\leq i_0\leq j_0+n-m_0-1$ such that
$a_{i_0}\cdots a_{i_0+N_0-1}> 2/3$ and $a_{i_0}\cdots a_{i_0+m_0-1}> 2/3$. Thus,  the unit vectors
 $$
   u^{ss}=\frac{A^{i_0-j_0}_{j_0}(u^{ss}_0)}{\|A^{i_0-j_0}_{j_0}(u^{ss}_0)\|}\in E^{ss}_{i_0},
\quad
\bar u^{c}=\frac{A^{i_0-j_0}_{j_0}(u^{c}_0)}{\|A^{i_0-j_0}_{j_0}(u^{c}_0)\|}\in E^{cs}_{i_0}\oplus E^{cu}_{i_0},
 $$
satisfy
\begin{equation}\label{e.eq2}
\|A^{N_0}_{i_0}(u^{ss})\|> \tfrac23 \|A^{N_0}_{i_0}(\bar u^{c})\|,
\end{equation}
\begin{equation}\label{e.eq3}
\|A^{m_0}_{i_0}(u^{ss})\|> \tfrac23 \|A^{m_0}_{i_0}(\bar u^{c})\|.
\end{equation}
Let $u^{cs}\in E^{cs}_{i_0}$ be any unit vector. Define
 $$
   u^c=(\bar u^c+a.u^{cs})/\|\bar u^c+a.u^{cs}\|.
 $$
Since the cocycle $(A_i)$ and its inverse are bounded by $C$, eq.~\eqref{e.eq2}  yields
 $$\begin{aligned}
 \|A^{N_0}_{i_0}(u^{ss})\| &> \tfrac23(1-a) \|A^{N_0}_{i_0}(u^{c})\|
    - \tfrac23 a\|A^{N_0}_{i_0}(u^{cs})\|\\
   & > \tfrac23(1-a) \|A^{N_0}_{i_0}(u^{c})\|
    - \tfrac23 aC^{2N_0}\|A^{N_0}_{i_0}(u^{ss})\|\\
 \end{aligned}$$
which gives item (b).
Note that $A^{m_0}_{i_0}(u^{c})$ decomposes as
 $$
  \|\bar u^c+a.u^{cs}\|^{-1}.\left(A^{m_0}_{i_0}(\bar u^{c})+a.A^{m_0}_{i_0}(u^{ss})\right).
   $$
From~\eqref{e.eq3} and~\eqref{e.eq1}, one gets
$$\|A^{m_0}_{i_0}(\bar u^{c})\|< \tfrac32 \|A^{m_0}_{i_0}(u^{ss})\|
< \eta a \|A^{m_0}_{i_0}(u^{cs})\|.$$
This implies item (a).
\end{proof}

\begin{lemma}\label{l.domination-symplectic}
For any integer $\widetilde N\geq1$, there exists an integer $N'\geq1$ such that
if the splitting $E^{ss}\oplus (E^{cs}\oplus E^{cu})$
is $\widetilde N$-dominated, then the cocycle $(A_i)$ has an $N'$-dominated splitting.
\end{lemma}
\begin{proof}
By $\widetilde N$-domination,
the angle between $E^{ss}$ and $E^c=E^{cs}\oplus E^{cu}$
is lower bounded and a variant of Corollary \ref{c.change-basis2} yields a conjugacy by a bounded, symplectic cocycle which sends $E^{ss}$ and $E^{c}=E^{cs}\oplus E^{cu}$ to the constant bundles
$\RR^{j}\times\{0\}^{d_0-j}$ and $\{0\}^j\times \RR^{d_0-2j}\times\{0\}^j$.
The cocycle $(A_i)$ is bounded by $C'$ and has the form:
 \begin{equation}\label{eq.reduc3}
   A_i=\begin{pmatrix} B_i^T & 0 & D_i\\
          0 & C_i & 0\\
          0 & 0 & B_i^{-1} \end{pmatrix},
 \end{equation}
where $B_i,D_i$ are $j\times j$ matrices and $C_i\in Sp(d_0-2j,\RR)$.

In the following we denote by $m(A)=\|A^{-1}\|^{-1}$ the co-norm of a matrix.
Since the change of coordinates is bounded, the $\widetilde N$-dominated splitting
gives uniform numbers $a>0$ and $b\in (0,1)$ such that
for any $n\geq 0$ and any $i\in\ZZ$, 
 $$
  \|B_{i+n}^T\dots B_{i+1}^T\|\leq
a b^n m(C_{i+n}\dots C_{i+1}).
 $$
Now, there is a constant $c$ (depending only on $d_0$) such that
 $$
   \|(B_{i+n}^{-1}\dots B_{i+1}^{-1})^{-1}\| \leq c
       \|((B_{i+n}^{-1}\dots B_{i+1}^{-1})^{-1})^T\|
        = c \|B_{i+n}^T\dots B_{i+1}^T\|.
 $$
Hence,
 $$
    \|(B_{i+n}^{-1}\dots B_{i+1}^{-1})^{-1}\| \leq ac b^n \|(C_{i+n}\dots C_{i+1})^{-1}\|^{-1}.
 $$
As each $C_i$ is symplectic, $JC_i^{-1}=C_i^TJ$, so we have  $\|(C_{i+n}\dots C_{i+1})^{-1}\|= \|C_{i+n}\dots C_{i+1}\|$ and
 $$
  \|C_{i+n}\dots C_{i+1}\| \leq acb^n m(B_{i+n}^{-1}\dots B_{i+1}^{-1}). 
 $$
This also gives
 $$
    \|B_{i+n}^T\dots B_{i+1}^T\|\leq
ca^2 b^{2n} m(B_{i+n}^{-1}\dots B_{i+1}^{-1}).
 $$
 Since the matrices $D_i$
are uniformly bounded, the cone-field criterion is satisfied
and gives a uniform dominated splitting between the bundle
$\RR^{d_0-j}\times \{0\}^{j}$ and a transverse bundle with dimension $j$.

The dominated splitting for the initial cocycle $(A_i)$ is obtained
by pulling back by the bounded conjugacy.
This shows that the splitting
$(E^{ss}\oplus E^c)\oplus E^{uu}$ is $N'$-dominated for some uniform
integer $N'$.
\end{proof}

\begin{proof}[Proof of the Proposition~\ref{p.BochiBonatti-reduced}]
Given $N_0,\varepsilon'$, the lemmas above give $\eta$, $\widetilde N$, $N'$.
Consider a cocycle $(A_i)$ as in the statement of the proposition.
One may assume that the splitting $E^s=E^{ss}\oplus E^{cs}$ is $N_0$-dominated
(otherwise there is nothing to show) and that 
 $E^{ss}\oplus (E^{cs}\oplus E^{cu})$ is not
$\widetilde N$-dominated (otherwise Lemma~\ref{l.domination-symplectic} provides a $N'$-dominated splitting).
Hence, Lemma~\ref{l.break-domination} gives integers $i_0\in \ZZ$, $m_0\geq N_0$ and some unit vectors
$u^c\in E^{cs}_{i_0}\oplus E^{cu}_{i_0}$, $u^{ss}\in E^{ss}_{i_0}$.

One chooses a subspace $\widehat E^{cs}_{i_0+m_0}$
in $E^{cs}_{i_0+m_0}\oplus E^{cu}_{i_0+m_0}$ that is $\eta$-close to $E^{cs}_{i_0+m_0}$ (with equal dimension)
and contains $A^{m_0}_{i_0}(u^c)$ .
Since $E^{cs}\oplus E^{cu}$ is (not necessarily strongly) dominated,
there exists an index $k<i_0$ such that
$$\widehat E^{cs}_{k+1}:=(A^{i_0+m_0-k-1}_{k+1})^{-1}(\widehat E^{cs}_{i_0+m_0})$$
is $\eta$-close to $E^{cs}_{k+1}$.
Applying Lemma~\ref{c.rectify} twice, one builds two
$\varepsilon'$-paths of perturbations $(A_{k}(t))$ and $(A_{i_0+m_0}(t))$ of $A_k$ and $A_{i_0+m_0}$ respectively such that
\begin{itemize}
\item[--] the restrictions to $E^{ss}$ of  $A_{k}(t)$ and  $A_{i_0+m_0}(t)$ are constant in $t$, and
\item[--] $A_{k}(E^{cs}_{k})=\widehat E^{cs}_{k+1}$ and
$A_{i_0+m_0}(\widehat E^{cs}_{i_0+m_0})=E^{cs}_{i_0+m_0+1}$.
\end{itemize}

One then chooses some large even multiple $\ell$ of the period of the initial cocycle and sets $A'_{i+s\ell}(t)=A_{i}(t)$
for each $s\in \ZZ$ and $i\in \{-\ell/2,\dots,\ell/2-1\}$.
Since $\ell$ is large, the Oseledets splitting $E^{ss}\oplus E^{cs}\oplus E^{cu}\oplus E^{uu}$ for the $\ell$-periodic cocycle $(A_i(t))$
can be followed continuously with $t$.
The items \ref{i.1}, \ref{i.3} and \ref{i.4} hold.

Item \ref{i.5} follows from the  Lemma~\ref{l.break-domination} (b), since $u^{ss}\in E^{ss}_{i_0}$, $u^c\in E^{cs}_{i_0}$ for the cocycle $(A_i(1))$ and since $A_{i_0+n}(1)=A_{i_0+n}$ for $n=0,\dots,N_0-1$.

Note also that the determinant along $E^{cs}$ of
$$A_{\ell/2-1}(t)\circ\dots\circ A_{\ell/2+1}\circ A_{\ell/2}(t)$$ changes only by a factor bounded independently from $\ell$ (the maps and the spaces of the bundle$E^{cs}$ are unchanged for indices not in $\{k,\dots,i_0+m_0\}$).
Since the spaces of the Oseledets splitting are $\omega$-orthogonal,
one can compose by symplectic maps
which act as the identity along $E^{ss}_j$ and as homotheties along $E^{cs}_j$
for some indices
$j=\ell/4,\dots,\ell/2-1$.
For these indices $E^{cs}$ coincides with the initial bundle $E^{cs}$,
hence is uniformly far from $E^{ss}$. Lemma~\ref{l.change-basis} controls the size of this perturbation.
One thus obtains $\varepsilon'$-paths of perturbations
$(A_{\ell/4}(t))$,\dots, $(A_{\ell/2-1}(t))$
such that the determinant of
$A_{\ell/2-1}(t)\circ\dots\circ A_{\ell/2}(t)$ along the continuation of the bundle $E^{cs}$
is constant in $t$. This implies the item~\ref{i.2} for the stable bundle.
The determinant along the stable and unstable bundles being inverse of each other, item~\ref{i.2} also holds for the unstable bundle.
\end{proof}

\section{Homoclinic tangencies}

This section extends to the conservative setting the following theorem of Gourmelon \cite{Gourmelon2010,Gourmelon2014} on the creation of homoclinic tangencies from a lack of dominated splitting.  We will follow the main steps of~\cite{Gourmelon2014} with the exception of the induction on the dimension. This induction is quite technical in the dissipative setting and difficult to adapt in the conservative setting. Avoiding it results in a simplification of Gourmelon's proof.

\begin{theorem}[Homoclinic tangency]\label{t.gourmelon}
For any {$d_0\geq1$}, $C>1$, $\varepsilon>0$, there exist $N,T\geq 1$ with the following property.
Consider
\begin{itemize}
\item[--] a diffeomorphism $f\in \Diff^1(M)$ {of a $d_0$-dimensional manifold $M$}
such that the norms of $Df$ and $Df^{-1}$ are bounded by $C$,
\item[--] a periodic saddle $\cO$ with period larger than $T$ such that {the splitting defined by the stable and unstable bundles is not an $N$-dominated splitting,} and
\item[--] a neighborhood $V$ of $\cO$.
\end{itemize}
Then there exists an $(\varepsilon, V, \cO)$-perturbation $g$ of $f$ and $p\in \cO$ such that
\begin{itemize}
\item[--] $W^s(p)$, $W^u(p)$ have a tangency $z$ whose orbit is contained in $V$, and
\item[--] the differential of $f$ and $g$ coincide along $\cO$.
\end{itemize}

Moreover if $\cO$ is homoclinically related to a periodic point $q$ for $f$, then the perturbation $g$
can be chosen to still have this property. If $f$ preserves a volume or a symplectic form, one can choose $g$ to preserve it also.
\end{theorem}

\begin{remark}\label{r.gourmelon}
As a consequence, one can also obtain a transverse intersection $z'$ between $W^s(\cO)$ and $W^u(\cO)$,
whose orbit is contained in $V$. This implies that $O$ belongs to a horseshoe of $g$ which is contained in $V$.
\end{remark}

\begin{proof}[Proof of Theorem~\ref{t.gourmelon}]
{As noted, the above theorem was proved in the dissipative setting by Gourmelon in \cite[Theorem 3.1]{Gourmelon2010} and~\cite[Theorem 8]{Gourmelon2014}. We consider the symplectic setting.}
One fixes $\eta>0$ given by Proposition~\ref{p.elementary} and $\chi>0$ much smaller than $\eta$.
\medbreak
\noindent{\it Step 1. Reduction.}
After a first perturbation,
using the linear statements (Proposition~\ref{p.real}, Theorem~\ref{t.angle} and Corollary~\ref{c.real})
together with the Franks' lemmas (Theorems~\ref{t.linearize} and~\ref{l.gourmelon})
one can assume the following properties:
\begin{itemize}
\item[--] a neighborhood of $\cO$ admits a chart $\psi$ such that $\psi\circ f \circ \psi^{-1}$ is linear near each point of $\cO$,
\item[--] $\cO$ has $d_0$ distinct real eigenvalues, with moduli in $(0,1/2)\cup (2,+\infty)$,
\item[--] the angle between stable and unstable spaces at some iterate $p\in \cO$ is smaller than $\chi^{d_0}$.
\end{itemize}
When $\cO$ is homoclinically related to a periodic point $q$, one fixes two transverse intersections
$z^s\in W^s(\cO)\cap W^u(q)$ and $z^u\in W^u(\cO)\cap W^s(q)$ and {choose the above perturbation to preserve these}.

\medbreak
\noindent{\it Step 2. Choice of iterates.}
Let $TM_{|\cO}=E^s_1\oplus \dots \oplus E^s_k\oplus
E^u_{1}\oplus\dots\oplus E^u_{d_0-k}$ be the invariant decomposition into one-dimensional bundles.
Since the angle between $E^s(p)$, $E^u(p)$ is smaller than $\chi^{d_0}$,
one can find $1\leq i\leq k$ and $1\leq j\leq d_0-k$ such that
\begin{itemize}
\item[a --] the angle $\theta$ between $E:=E^s_1\oplus \dots\oplus  E^s_i$ and
$F:=E^u_1\oplus \dots \oplus E^u_j$ is in $(0,\chi^{i+j})$;
\item[b --] the angles between $E$ and $F':=E^u_1\oplus \dots E^u_{j-1}$
and between $E':=E^s_1\oplus \dots \oplus E^s_{i-1}$ and $F$
are larger than ${\theta}/{\chi}$ (by convention the angle of the zero subspace with any other subspace is infinite).
\end{itemize}
{Indeed, setting $\theta_{ij}:=\angle(E^s_1\oplus\dots\oplus E^s_i,E^u_1\oplus\dots\oplus E^u_j)$,
one builds inductively a sequence $\mathcal S$ of pairs $(i,j)$ such that the angle $\theta_{i,j}$
satisfies (a). The initial pair is $(k,d_0-k)$.
If $(i,j)\in \mathcal{S}$ and if $\theta_{i,j}$ does not satisfies the condition (b),
then the new pair in the sequence is either $(i-1,j)$ or $(i,j-1)$: one of these two pairs has to satisfy (a).
The last pair obtained during this construction satisfies (b) as required.}

For $\rho>0$ small, one can choose
$u\in E$ and $v\in F$ with norm $\rho$ such that $\|u-v\|<\theta.\rho$.
Since the angle between $v$ and $E'$ is larger than $\theta/\chi$
(and since $\chi$ is small), the orthogonal projection of $u$ on $E\cap (E')^\perp$
has norm larger than $\theta\rho/(2\chi)$.

Since $\eta\gg \chi$, the quantity $\theta\rho/\eta$
is much smaller than $\theta\rho/(2\chi)$.
Consequently, the orthogonal projection on $E\cap (E')^\perp$
of the ball $B$ centered at $u$ and with radius $\|u-v\|/\eta$ has diameter {much} smaller
than the projection of $u$.
Since $E'$ and $E$ are preserved by $Df^{\pi(\cO)}$, since $E/E'$ is one-dimensional,
and since the eigenvalue of $Df^{\pi(\cO)}$ on the quotient $E/E'$
belongs to $(-1/2,1/2)$, one deduces that the forward orbit of $u$ does not intersect
the ball $B$.
Similarly the backward iterate of $v$ does not intersect the ball $B$. 
Using the linearity of $f$ in the chart near $\cO$, this also holds for the $f$-orbits of $u$ and $v$ as points in $M$. 
\medskip

\noindent{\it Step 3. The homoclinic tangency.}
Proposition~\ref{p.elementary} gives an $\varepsilon$-perturbation $g$
of $f$ such that $f^{-1}\circ g$ is supported in $B$ and such that
$g(v)=f(u)$. One deduces that the orbit of $v$ for $g$ is homoclinic
to $\cO$ (and contained in a small neighborhood of $\cO$).
Since the angle between $E$ and $F$ is small,
with Theorem~\ref{t.linearize}, one can ensure that $Dg(v)(F)$ is tangent to $E$
along the line directed by $u$. One has obtained a homoclinic tangency.
\medskip

Assume now that $\cO$ is homoclinically related to some point $q$
through the orbits of $z^s,z^u$.
By construction, the orthogonal projection of $B\cap E$ on $E\cap (E')^\perp$
has a small diameter in comparison to the distance to the origin.
Since the eigenvalue of $Df^{\pi(\cO)}$ on the quotient $E/E'$
belongs to $(-1/2,1/2)$, one can choose $B$ (i.e. the norm $\rho$)
so that $B$ is disjoint from the orbit of $z^s$. The same can be done for $z^u$.
This proves that the homoclinic connection with $q$ is preserved.

Finally, one can apply Theorem~\ref{t.linearize} so that the {differential} of $f$ and of the
perturbed system coincide along the orbit of $\cO$ while keeping the homoclinic connection.
\end{proof}

\section{Proof of Theorem \ref{t.horseshoes}}

The proof can be summarized as follows: Let $\Lambda$ be an infinite, transitive compact set for the diffeomorphism $f$. It will be approximated by a periodic orbit created by Pugh's closing lemma. This periodic orbit cannot have strong domination. Theorem \ref{t.gourmelon} (more precisely, Remark \ref{r.gourmelon}) will yield a horseshoe close to the periodic orbit, provided this orbit is a saddle. 

To reduce to this case, we will use a different argument in the conservative and dissipative cases. In the conservative setting, we will perturb the differential and linearize to build saddles (with higher periods). In the dissipative setting (where, generically, there are weak sinks or sources accumulating on $\Lambda$), we will rely on the following result, see~\cite{pliss}:

\begin{proposition}[Pliss]\label{p.pliss}
For any $d_0\geq 2$, $C\geq 1$ and $\varepsilon>0$, there exists $N,T\geq 1$ such that,
if $g$ is a diffeomorphism with $Dg, Dg^{-1}$ bounded by $C$ and if $\cO$ is an
attracting periodic orbit with period $\ell$ larger than $T$, then
\begin{itemize}
\item[--] either $\prod_{x\in \cO} \|Dg^N(x)\|\leq 2^{-\ell}$,
\item[--] or there exists a diffeomorphism $C^1$-close to $g$
which preserves $\cO$ and such that $\cO$ is a hyperbolic saddle
(it has both stable and unstable eigendirections.).
\end{itemize}
\end{proposition}

\medskip

\begin{proof}[Proof of Theorem \ref{t.horseshoes}]
We fix $\varepsilon>0$ small
and $C>0$ which bounds the norms of $Dg$ and $Dg^{-1}$
for any diffeomorphism that is $\varepsilon$-close to $f$ for the $C^1$-topology.
The dimension of $M$ is $d_0\geq 2$.
Theorem~\ref{t.gourmelon} provides integers $N,T\geq1$ given $d_0,C,\varepsilon/2$.

We first perturb $f$ to create a periodic orbit approximating $\Lambda$.

\begin{lemma}
For any $\varepsilon_1>0$, there exists $f_1$ with $d_{C^1}(f,f_1)<\varepsilon_1$
and a periodic orbit $\cO$ of $f_1$ that is $\varepsilon_1$-close to $\Lambda$ for the Hausdorff distance.
\end{lemma}
\begin{proof}
Since $\Lambda$ is transitive, there exists a point $x\in \Lambda$
whose orbit is dense in $\Lambda$. Since $\Lambda$ is infinite, $x$
is not periodic. From Pugh's closing lemma~\cite{pugh}, there exists $N\geq 1$
with the following property:
for any neighborhood $U$ of $x$, there exists a diffeomorphism $f_1$
having a periodic orbit $\cO$ which intersects $U$ such that:
\begin{itemize}
\item[--]  $d_{C^1}(f,f_1)<\varepsilon_1$,
\item[--] $f_1=f$ on $M\setminus (U\cup\dots\cup f^{N-1}(U))$,
\item[--] $\cO\subset \cO(x)\cup U\cup\dots\cup f^{N-1}(U)$.
\end{itemize}
When $f$ is conservative, $f_1$ can still be chosen conservative
(the closing lemma is still valid for conservative systems~\cite{PR}).

Since the orbit of $x$ is dense in $\Lambda$, if the diameter of $U$ is small enough, 
the orbit $\cO$ (which intersects $U$) intersects the $\varepsilon_1$-neighborhood of any point of $\Lambda$.
From the third item above, it is contained in the $\varepsilon_1$-neighborhood of $\Lambda$.
Hence $\cO$ and $\Lambda$ are $\varepsilon_1$-close.
\end{proof}

We now turn the periodic orbit $\cO$ into a hyperbolic saddle.

\begin{lemma}
For any $\varepsilon_2>0$, there exists $f_2$ with $d_{C^1}(f,f_2)<\varepsilon_2$
and a hyperbolic orbit $\widetilde \cO$ of $f_2$ that is $\varepsilon_2$-close to $\Lambda$ for the Hausdorff distance
and has saddle type (i.e., both stable and unstable eigenvalues).
\end{lemma}

\begin{proof}
Since $\Lambda$ is non-periodic, the period of the orbit $\cO$ given by the previous lemma is arbitrarily large,
provided $\varepsilon_1$ has been chosen small enough.
The proof then differs in the conservative and in the dissipative cases.

\paragraph{\it In the dissipative case,}
Proposition~\ref{p.simple} and Franks' lemma (e.g., Theorem \ref{t.linearize}) provide a perturbation of $f_1$ so that $\cO$ becomes a hyperbolic periodic orbit
(see also Kupka-Smale theorem~\cite{kupka,smale-kupka}), which is a saddle, a sink or a source.
Let $N_2,T_2\geq1$ be integers defined by Proposition~\ref{p.pliss} given $\frac 1 2 \varepsilon_2>0$.
Since the period of $\cO$ can be chosen arbitrarily large, Proposition~\ref{p.pliss} will apply whenever $\cO$ is attracting.
If, after a $\frac 1 2 \varepsilon_2$-perturbation, one gets a saddle periodic orbit $\widetilde \cO=\cO$, we are done. Otherwise, we can assume that all such perturbations have a sink (the case of a source is left to the reader). Proposition \ref{p.pliss} then yields that
the average of $\log \|Df_1^{N_2}\| $ for the measure  $\mu_\cO$ defined by $\cO$ is smaller than $-\log 2$.
Since $f_1$ can be chosen arbitrarily close to $f$ and $\cO$ to $\Lambda$,
one can take an accumulation point of the measures $\mu_\cO$ and an ergodic component $\mu$. It will be carried on $\Lambda$ with
$\int \log \|Df^{N_2}\| d\mu\leq -\log 2$. In particular $\mu$-almost every point has a stable manifold
of dimension $d_0$ (see for instance~\cite[Theorem 3.11]{ABC}) and $\mu$ is a sink.
This is  a contradiction since $\Lambda$ is transitive and infinite.

\paragraph{\it In the conservative case,}
Proposition~\ref{p.simple} provides a perturbation $f_1$ such that the eigenvalues of $\cO$
are simple and that the eigenvalues with modulus $1$ are roots of the unity.
Theorem~\ref{t.linearize} linearizes the dynamics in a neighborhood of $\cO$ by a further perturbation.
This implies that there exists a periodic orbit $\widetilde \cO$ contained in an arbitrarily small
neighborhood of $\cO$ and whose period $\widetilde \ell$ is a multiple of $\ell$,
such that $Df^{\widetilde \ell}_{|\widetilde \cO}$ is the product of a hyperbolic linear map and the identity. Since the identity can be turned into a hyperbolic linear map
by a small perturbation in the symplectic group, a further small perturbation
provided by Theorems~\ref{t.linearize} ensures that $\widetilde \cO$ is hyperbolic as required.
\end{proof}

The existence of a $N$-dominated splitting passes to the limit when one considers sequences
of diffeomorphisms and of invariant compact sets.
Consequently, if $\varepsilon_2\in (0,\varepsilon/2)$ is chosen small enough,
the previous lemma provides a diffeomorphism $f_2$
with a periodic orbit $\cO$ whose decomposition of the tangent bundle
into stable and unstable spaces is not $N$-dominated and whose period is larger than $T$.

Theorem~\ref{t.gourmelon} and remark~\ref{r.gourmelon}
then build a diffeomorphism $g$ that is $\varepsilon/2$-close to $f_2$
such that $\cO$ has a transverse homoclinic orbit in an arbitrarily small neighborhood of $\cO$.
In particular, this gives a horseshoe in an arbitrarily small neighborhood of $\Lambda$, as required.
\end{proof}

\noindent
\emph{J\'er\^ome Buzzi}\\
{\small LMO, CNRS - UMR 8628, Universit\'e Paris-Sud 11, 91405 Orsay, France.}\\

\noindent
\emph{Sylvain Crovisier}\\
{\small LMO, CNRS - UMR 8628, Universit\'e Paris-Sud 11, 91405 Orsay, France.}\\

\noindent
\emph{Todd Fisher}\\
{\small Department of Mathematics, Brigham Young University,
Provo, UT 84602, USA.}

\end{document}